\documentclass[12pt]{article}
\usepackage{amsfonts,amscd,a4}
\usepackage{color}
\usepackage{graphicx}
\usepackage{theorem}
\usepackage{amssymb}

\newtheorem{theorem}{Theorem}
\newtheorem{proposition}[theorem]{Proposition}

\newtheorem{corollary}[theorem]{Corollary}
\newtheorem{definition}[theorem]{Definition}
{\theorembodyfont{\rm}

}
\newenvironment{proof}[1][Proof]{\noindent\textbf{#1.} }{\hfill
\rule{0.5em}{0.5em}}

\newcommand{\cB}{{\mathcal B}}

\newcommand{\cE}{{\mathcal E}}

\newcommand{\cH}{{\mathcal H}}
\newcommand{\cI}{{\mathcal I}}

\newcommand{\cK}{{\mathcal K}}
\newcommand{\cL}{{\mathcal L}}
\newcommand{\cM}{{\mathcal M}}

\newcommand{\cR}{{\mathcal R}}
\newcommand{\cS}{{\mathcal S}}

\newcommand{\cV}{{\mathcal V}}

\newcommand{\sC}{{\mathbb C}}

\newcommand{\sN}{{\mathbb N}}

\newcommand{\sR}{{\mathbb R}}

\newcommand{\sT}{{\mathbb T}}

\newcommand{\sZ}{{\mathbb Z}}

\newcommand{\diag}{\mbox{\rm diag} \,}

\newcommand{\val}{{\rm val}}
\begin{document}
\title{Pseudodifferential operators on periodic graphs}
\author{V. Rabinovich, S. Roch}
\date{}
\maketitle
\begin{abstract}
The main aim of the paper is Fredholm properties of a class of bounded
linear operators acting on weighted Lebesgue spaces on an infinite metric graph
$\Gamma$ which is periodic with respect to the action of the group $\sZ^n$.
The operators under consideration are distinguished by their local behavior:
they act as (Fourier) pseudodifferential operators in the class $OPS^0$ on
every open edge of the graph, and they can be represented as a matrix Mellin
pseudodifferential operator on a neighborhood of every vertex of $\Gamma$.
We apply these results to study the Fredholm property of a class of singular
integral operators and of certain locally compact operators on graphs.
\end{abstract}
\section{Introduction}
Schr\"{o}dinger operators on combinatorial and quantum graphs have attracted a
lot of attention in the last time due to their interesting properties and
existing and expected applications in nano-structures (see, for instance,
\cite{Berk,Kuch1,Kuch3,Kuch4,Kuch5} and the references cited there).
The present paper is devoted to a quite general class of bounded linear
operators acting on weighted Lebesgue spaces on an infinite metric graph
$\Gamma$ which is periodic with respect to the action of the group $\sZ^n$.
Our main emphasis is on Fredholm properties of these operators.

More precisely, the operators under consideration are distinguished by their
local behavior: they act as (Fourier) pseudodifferential operators in the class
$OPS^0$ on every open edge of the graph, and they can be represented as a matrix
Mellin pseudodifferential operator on a neighborhood of every vertex of
$\Gamma$. The appearance of Mellin convolution operators in this context is
quite natural: near a vertex of the graph, a pseudodifferential operator can be
written as a sum of a singular integral operator on a system of rays and a
locally compact operator, and every operator of this form corresponds to a
matrix Mellin convolution or, more general, a matrix Mellin pseudodifferential
operator (see, for instance, Chap. 4 in \cite{RSS}, Chap. 4.6 in \cite{RRS},
\cite{RIzv} and references cited there). Mellin convolution operators were used
in \cite{OP} to study pseudodifferential operators on {\em finite} graphs.

When studying Fredholm properties of general (non-periodic) operators, the most
challenging part is to understand their local invertibility at infinity. For
this goal, we use the limit operators method (see \cite{RRS} for an overview)
which will allow us to reduce the local invertibility at infinity to the
invertibility of every single operator in a family of periodic operators, the
so-called limit operators. The limit operators method already proved to be a
very effective tool for the investigation of essential spectra of operators on
combinatorial periodic graphs in \cite{RR1,RR2} and of essential spectra of
Schr\"{o}dinger, Dirac and Klein-Gordon operators on $\sR^n$ in \cite{RJMP}.

The paper is organized as follows. In Section 2 we recall some auxiliary
material concerning Wiener algebras on $\sZ^n$, (Fourier) pseudodifferential
operators on $\sR$, and matrix Mellin pseudodifferential operators on $\sR_+ =
[0, \, \infty)$. In Section 3 we introduce a class of pseudodifferential
operators on a periodic graph and establish necessary and sufficient
conditions for their Fredholmness. Our basic tool to study the Fredholm
property is Simonenko's local principle (see \cite{Sim} and Section
2.5 in \cite{RSS}). The limit operators method can be used to relate the local
invertibility at infinity to the invertibility of periodic operators, to which
then a version of Floquet theory can be applied. In Section 4 we consider two
applications:
\begin{itemize}
\item[(i)]
the Fredholm property of singular integral operators $A = aI + b \cS_{\Gamma,
\phi}$ on a periodic graph $\Gamma \subset \sR^2$, where $a, \,b$ are
certain bounded slowly oscillating and piecewise continuous functions with
discontinuities only at the vertices of $\Gamma$, and where $\cS_{\Gamma, \phi}$
is the singular integral operator
\[
(\cS_{\Gamma, \phi} u)(x) = \frac{1}{\pi i} \int_\Gamma
\frac{\phi(x,x-y)u(y)} {y-x} dy, \quad x \in \Gamma,
\]
with $\phi \in C^\infty (\Gamma \times \sR^2)$ a function the decaying behavior
of which will be specified below;
\item[(ii)]
the Fredholm property of integral operators $A = aI + bT$ on a periodic
graph $\Gamma \subset \sR^2$ where $a, \, b$ are bounded, uniformly
continuous and slowly oscillating functions on $\Gamma$ and $T$ is an integral
operator of the form
\[
(Tu)(x) = \int_\Gamma k(x-y) u(y) dy, \quad x \in \Gamma,
\]
where $k: \sR^2 \to \sC$ is a continuous function with the property that
there are positive constants $C$ and $\varepsilon$ such that
\[
|k(z)| \le C (1 + |z|)^{-2 - \varepsilon} \quad \mbox{for all} \; z \in \sR^2.
\]
\end{itemize}
This work was supported by CONACYT Project 81615 and DFG Grant Ro 1100/8-1.
\section{Auxiliary material}
\subsection{Wiener algebras on $\sZ^n$}
Given a complex Banach space $X$, let $\cB(X)$ denote the Banach algebra of all
bounded linear operators on $X$ and $\cK(X)$ the closed ideal of $\cB(X)$ of
all compact operators. For $A \in \cB(X)$, we write $sp_X A$ for the spectrum
and $ess\,sp_X A$ for the essential spectrum of $A$, i.e., for the set of all
$\lambda \in \sC$ such that the operator $A - \lambda I$ is not Fredholm on $X$.
Further we let $l^p (\sZ^n, \, X)$ stand for the Banach space of all functions
$u : \sZ^n \to X$ with
\[
\|u\|^p_{l^p (\sZ^n, \, X)} := \sum_{x \in \sZ^n} \|u(x)\|_X^p < \infty
\]
if $p \in [1, \, \infty)$ and
\[
\|u\|_{l^\infty (\sZ^n, \, X)} := \sup_{x \in \sZ^n} \|u(x)\|_X < \infty.
\]
On $l^p (\sZ^n, \, X)$, we consider operators of the form
\begin{equation} \label{e1neu}
A = \sum_{\alpha \in \sZ^n} a_\alpha \tau_\alpha
\end{equation}
where $a_\alpha \in l^\infty (\sZ^n, \, \cB(X))$ and $\tau_\alpha$ is the
operator of shift by $\alpha \in \sZ^n$,
\[
(\tau_\alpha u)(x) := u(x-\alpha), \quad x \in \sZ^n.
\]
We say that the operator $A$ in (\ref{e1neu}) belongs to the Wiener algebra
$W(\sZ^n, \, X)$ if
\[
\|A\|_{W(\sZ^n, X)} := \sum_{\alpha \in \sZ^n} \|a_\alpha\|_{\cB(X)} <
\infty.
\]
It is well known that $W(\sZ^n, \, X) \subset \cB(l^p(\sZ^n, \, X))$ for every
$p\in [1, \, \infty]$ and that $sp_{l^{p}(\sZ^n, X)} A$ does not depend on
$p$ if $A$ belongs to $W(\sZ^n, X)$.

Finally, we denote the operator of multiplication by a function $f$ by $fI$ and
let $\cH$ stand for the set of all sequences $h : \sN \to \sZ^n$ which tend
to infinity in the sense that, for every $R > 0$ there is an $m_0$ such that
$|h(m)| > R$ for all $m \ge m_0$.
\begin{definition} \label{def1}
Let $A \in \cB (l^p (\sZ^n, \, X))$. An operator $A^h \in \cB (l^p (\sZ^n, \,
X))$ is called the {\em limit operator of $A$ defined by the sequence} $h \in
\cH$ if, for every function $\varphi$ on $\sZ^n$ with finite support,
\[
\lim_{m \to \infty} \|(\tau_{h(m)}^{-1} A \tau_{h(m)} - A^h) \varphi
I\|_{\cB(l^p(\sZ^n, \, X))} = 0
\]
and
\[
\lim_{m \to \infty} \|\varphi (\tau_{h(m)}^{-1} A
\tau_{h(m)} - A^{h})\|_{\cB (l^p(\sZ^n, \, X))}= 0.
\]
We denote the set of all limit operators of $A$ by $Lim (A)$.
\end{definition}
We say that an operator $A \in \cB(l^p (\sZ^n, \, X))$ is rich if every sequence
$h \in \cH$ has a subsequence $g$ such that the limit operator $A^g$ of $A$
with respect to $g$ exists. For $r > 0$, let $\chi_r$ denote the characteristic
function of the set $\{ x \in \sZ^n : |z| > r \}$.
\begin{definition} \label{d1}
The operator $A \in \cB (l^p (\sZ^n, \, X))$ is called {\em locally invertible
at infinity} if there exist an $r > 0$ and operators $L_r, \, R_r \in \cB (l^p
(\sZ^n, \, X))$ such that $L_r A \chi_r I = \chi_rI$ and $\chi_r A R_r =
\chi_rI$.
\end{definition}
The following is Theorem 2.5.7 in \cite{RRS}.
\begin{proposition} \label{prop1}
Let $A \in W(\sZ^n,X)$ be a rich operator. Then $A$ is locally invertible at
infinity on $l^p (\sZ^n, \, X)$ for some $p \in (1, \infty)$ if and only every
limit operator $A^h$ of $A$ is invertible on one of the spaces
$l^p (\sZ^n,\, X)$ with $p \in [1, \, \infty]$.
\end{proposition}
\subsection{Pseudodifferential operators on $\sR$}
The theory of pseudodifferential operators is developed in several textbooks,
e.g. \cite{Shubin,Taylor}. Here we only fix the notation and collect some
facts for later reference. For proofs of Propositions \ref{pr1} and
\ref{prop2} see \cite{Taylor1} and Chapter 3 of \cite{Taylor}, respectively.

Let $\langle \xi \rangle := (1 + |\xi|^2)^{1/2}$. We say that $a$ is a
symbol in the class $S^m$ if $a \in C^\infty (\sR \times \sR)$ and
\[
|a|_{k,l} := \sup_{(x, \xi) \in \sR \times \sR} \sum_{\alpha \le k, \, \beta
\le l} |\partial_x^\beta \partial_\xi^\alpha a(x, \, \xi)| \langle \xi \rangle
^{\alpha -m} < \infty
\]
for all $k, \, l \in \sN_{0}$. To each symbol $a$, we correspond a (Fourier)
pseudodifferential operator ($\psi$do for short) $Op(a)$ which acts on $u \in
C_0^\infty (\sR)$ by
\[
(Op(a)u)(x) := \frac{1}{2 \pi} \int_{\sR} a(x, \, \xi) \hat{u}(\xi) e^{ix\xi}
d\xi, \quad x \in \sR,
\]
where $\hat{u}$ stands for the Fourier transform of $u$. By $OPS^{0}$ we
denote the class of all $\psi$dos with symbol in $S^0$. We will need the
following properties of $\psi$dos.

\begin{proposition} \label{pr1}
$(a)$ Let $a \in S^0$. Then $Op(a)$ is bounded on $L^p (\sR)$ for $p \in
(1, \, \infty)$ and
\[
\|Op(a)\|_{\cB (L^p (\sR))} \le C |a|_{2,2}
\]
with a constant $C$ independent of $a$. \\[1mm]
$(b)$ Let $a_1 \in S^{m_1}$ and $a_2 \in S^{m_2}$. Then $Op(a_1) Op(a_2) =
Op(c)$ with $c \in S^{m_1+m_2}$ and $a_1a_2 - c \in S^{m_1 + m_2-1}$.
\end{proposition}
We say that an operator $A \in \cB (L^p (\sR))$ is {\em locally Fredholm at}
$x_0 \in \sR$ if there are an open neighborhood $U$ of $x_0$ and operators
$R_{x_0}, \, L_{x_0} \in \cB(L^p (\sR))$ such that
\[
L_{x_0} A \chi_U I - \chi_U I, \quad \chi_U A R_{x_0} - \chi_U I \in \cK( L^p
(\sR)).
\]
\begin{proposition} \label{prop2}
Let $a \in S^0$. Then $Op(a)$ is locally Fredholm at $x_0 \in \sR$
if and only if $Op(a)$ is elliptic at $x_0$, that is if $\liminf_{\xi \to
\infty} |a(x_0, \, \xi)| > 0$.
\end{proposition}
\subsection{Mellin pseudodifferential operators on $\sR_+$} \label{Sec 3}
Set $\sR_+ := (0, \, \infty)$. We say that a matrix function $a =
(a_{ij})_{i,j=1}^{n} : \sR_+ \times \sR \to \sC^{n \times n}$ belongs to $\cE
(n)$ if every entry $a_{ij}$ is in $C^\infty (\sR_+ \times \sR)$ and
\begin{equation} \label{2.2'}
|a|_{l_1, l_2} := \max_{1 \le i,j \le n} \sup_{(r,\lambda) \in \sR_+ \times \sR}
\sum_{\alpha \le l_1, \, \beta \le l_2} |(r \partial_r)^\beta \partial_\lambda
^\alpha a_{ij} (r, \, \lambda)| \langle \lambda \rangle^\beta < \infty
\end{equation}
for all $l_1, \, l_2 \in \sN_{0}$. Let $a \in \cE(n)$. Then the operator
\begin{equation} \label{2.3}
(op(a)u)(r) := \frac{1}{2 \pi} \int_\sR \int_{\sR_+} a(r, \, \lambda) \, (r\rho
^{-1})^{i\lambda} u(\rho) \rho^{-1} d\rho d\lambda, \quad r \in \sR_+,
\end{equation}
acting on $u \in C_0^\infty (\sR_+, \, \sC^n)$ is called the {\em Mellin
pseudodifferential operator} (Mellin $\psi$do for short) with symbol $a \in
\cE(n)$. The class of all Mellin $\psi$dos with symbol in $\cE(n)$ is denoted by
$OP\cE(n)$.

A function $a \in \cE(n)$ is said to be {\em slowly oscillating at the point}
$0$ if
\begin{equation} \label{2.1}
\lim_{r \to +0} \sup_{\lambda \in \sR} |(r\partial_r)^\beta
\partial_\lambda^\alpha a_{ij}(r, \, \lambda)| \langle \lambda \rangle^\alpha
= 0
\end{equation}
for all $\alpha \in \sN_0$ and $\beta \in \sN$. We let $\cE_{sl}(n)$ denote
the set of all functions which are slowly oscillating at $0$ and write
$\cE_{0}(n)$ for the set of all functions $a \in \cE(n)$ which satisfy
(\ref{2.1}) for all $\alpha, \, \beta \in \sN_0$. The corresponding classes of
Mellin $\psi$dos are denoted by $OP\cE_{sl}(n)$ and $OP\cE_{0}(n)$,
respectively.

Mellin $\psi$dos are pseudodifferential operators on the multiplicative group
$\sR_+$ with respect to the invariant measure $d\mu = \frac{dr}{r}$. They can
be obtained from (Fourier) $\psi$dos on $\sR$ by the change of variables $\sR
\ni x \mapsto r=e^{-x} \in \sR_+$, which transforms the point $+ \infty$ to the
point $0$. Consequently, the main properties of Mellin $\psi$dos follow
immediately from the corresponding properties of (Fourier) $\psi$dos on $\sR$.

Let $L^p (\sR_+, \, d\mu, \, \sC^n)$ denote the space of all measurable
functions $u : \sR_+ \to \sC^n$ with
\[
\|u\|_{L^p (\sR_+, \, d\mu, \, \sC^n)}^p := \int_{\sR_+} \|u(r)\|_{\sC^n}^p
d\mu < \infty.
\]
The following results can be found in \cite{RJMP}.
\begin{proposition} \label{p2.1}
Let $a \in \cE(n)$ and $p \in (1, \, \infty)$. Then the Mellin $\psi$do $op(a)$
is bounded on $L^p (\sR_+, \, d\mu, \, \sC^n)$ and
\begin{equation} \label{2.4}
\|op(a)\|_{\cB(L^p (\sR_+, \, d\mu, \, \sC^n))} \le C |a|_{2,2}
\end{equation}
with a constant $C$ independent of $a$.
\end{proposition}
\begin{proposition} \label{p2.2}
$(a)$ Let $a, \, b \in \cE(n)$. Then $op(a)op(b) = op(c) \in OP\cE(n)$, where
\begin{equation} \label{2.5}
c(r, \, \lambda) = \frac{1}{2 \pi} \int_{\sR} \int_{\sR_+} a(r, \, \lambda
+ \eta) b(r\rho, \, \lambda) \rho^{-i\eta } d\rho d\eta.
\end{equation}
$(b)$ Let $a \in \cE(n)$ and consider the operator $op(a)$ as acting on
$L^p (\sR_+, \, d\mu, \, \sC^n)$. Then $op(a)^\ast = op(b) \in \cE(n)$,
where
\begin{equation} \label{2.7}
b(r, \, \lambda) = \frac{1}{2 \pi} \int_{\sR} \int_{\sR_+} a^\ast (r\rho, \,
\lambda + \eta) \rho^{-i\eta} d\rho d\eta.
\end{equation}
\end{proposition}
Here $op(a)^\ast$ stands for the adjoint operator and $a^\ast$ for the
adjoint matrix function. The integrals in $(\ref{2.5})$ and $(\ref{2.7})$ are
understood as oscillatory integrals.
\begin{proposition} \label{p2.3}
$(a)$ Let $a, \, b \in \cE_{sl}(n)$. Then $op(a)op(b) = op(c)$ where $c \in
\cE_{sl}(n)$ and $c - ab \in \cE_0(n)$. \\[1mm]
$(b)$ Let $a \in \cE_{sl}(n)$ and consider the operator $op(a)$ as acting on
$L^p (\sR_+, \, d\mu, \, \sC^n)$. Then $op(a)^\ast = op(b)$ where $b \in
\cE_{sl}(n)$ and $b - a^\ast \in \cE_0(n)$.
\end{proposition}
In what follows, we will consider Mellin $\psi$dos on the weighted Lebesgue
spaces $L^p (\sR_+, \, w, \, d\mu )$ of all measurable functions with norm
\[
\|u\|_{L^p (\sR_+, \, w, \, d\mu)} := \|wu\|_{L^p(\sR_+, \, d\mu)},
\]
where the weight $w$ is of the form $w = \exp \sigma$ with a function $\sigma$
which satisfies the conditions
\begin{equation} \label{2.8}
\lim_{r \to +0} \left( r \frac{d}{dr} \right)^2 \sigma(r) = 0 \quad
\mbox{and} \quad \sup_{r \in \sR_+} \left| \left( r \frac{d}{dr} \right)^k
\sigma(r)
\right| < \infty
\end{equation}
for every $k \in \sN$. Moreover we assume that there is an interval $\cI = (c,
\, d)$ which contains 0 such that the function $\varkappa_\sigma (r) :=
r \sigma^\prime (r)$ satisfies
\begin{equation} \label{2.9}
c < \liminf_{r \in \sR_+} \varkappa_\sigma (r) \le \limsup_{r \in \sR_+}
\varkappa_\sigma (r) < d.
\end{equation}
We let $\cR(\cI)$ denote the collection of all weights $w$ such that
(\ref{2.8}) and (\ref{2.9}) hold. By $\cE(n, \, \cI)$ we denote the set of all
symbols $a \in \cE(n)$ such that the function $a(\cdot, \, \lambda )$ can be
analytically extended with respect to $\lambda$ into the strip $\Pi := \{\lambda
\in \sC : \mathfrak{I} (\lambda) \in \cI\}$ and this continuation satisfies
\[
\sup_{(r, \, \lambda) \in \sR_+ \times \Pi} |(r \partial_r)^\beta
\partial_\lambda^\alpha a_{ij} (r, \, \lambda)| < \infty.
\]
The corresponding class of Mellin $\psi$dos with analytical symbol is
denoted by $OP\cE(n, \, \cI)$.
\begin{proposition} \label{pr2}
Let $a \in \cE(n,\cI)$, $w \in \cR(\cI)$ and $p \in (1, \, \infty)$. Then the
operator $op(a)$ is bounded on $L^p (\sR_+, \, w, \, d\mu)$, and
\[
\|op(a)\|_{\cB (L^p(\sR_+, \, w, \, d\mu, \sC^n))} \le C |a|_{2,4}
\]
with a constant $C$ independent of $a$.
\end{proposition}
\begin{proposition} \label{Pr1.2}
Let $a \in \cE_{sl}(n, \, \cI)$ and $w = \exp \sigma \in \cR (\cI)$. Then $w
op(a)
w^{-1} = op(b)$ with $b \in \cE_{sl}(n)$ and
\begin{equation} \label{2.12}
b(r, \, \lambda) = a(r, \, \lambda + i\varkappa_\sigma (r)) + q(r, \, \lambda)
\end{equation}
where $q \in \cE_0(n)$.
\end{proposition}
Now we turn to local invertibility properties of Mellin pseudodifferential
operators. We say that an operator $A \in \cB (L^p (\sR_+, \, w, \, d\mu, \,
\sC^n))$ is {\em locally invertible at the point} $0$, if there are an $r > 0$
and operators $L_r$ and $R_r$ such that
\[
L_r A \chi_rI = \chi_r I \quad \mbox{and} \quad \chi_r A R_r = \chi_r I
\]
where $\chi_r$ refers to the characteristic function of the interval $[0, \,
r]$.
\begin{proposition} \label{prop3}
Let $a \in \cE_{sl}(n, \, \cI)$, $w = \exp \sigma \in \cR(\cI)$, and consider
$op(a)$ as
acting on $L^p (\sR_+, \, w, \, d\mu, \, \sC^n)$. Then $op(a)$ is locally
invertible at the point $0$ if and only if
\[
\liminf_{r \to +0} \inf_{\lambda \in \sR} |\det a(r, \, \lambda +
i\varkappa_\sigma (r))| > 0.
\]
\end{proposition}
\section{Operators on periodic graphs}
\subsection{Periodic graphs}
By a directed (combinatorial) graph (or digraph for short) we mean a pair
$\Gamma_{comb} = (\cV, \, \cE)$ consisting of a countably infinite set $\cV$
of vertices and a set $\cE \subseteq \cV \times \cV$ of edges. We think of $e =
(v, \, w) \in \cE$ as the oriented edge which starts at $v$ and ends at $w$. We
only consider digraphs without loops and without multiple edges, i.e., $\cE$
does not contain pairs of the form $(v, \, v)$, and if $(v, \, w) \in \cE$, then
$(w, \, v) \notin \cE$. Given a digraph, there is a related undirected
(combinatorial) graph, which arises by ignoring the orientation. Formally, the
undirected graph related with $\Gamma_{comb}$ is the pair $(\cV, \, \cE_{ud})$
where $\cE_{ud}$ is the set of all subsets of $\cV$ consisting of two elements $v, \, w$
such that $(v, \, w) \in \cE$. We say that $e = \{v, \, w\}$ connects the
vertices $v$ and $w$ and call $v$ and $w$ the endpoints of the edge $e$. For $v
\in \cV$, let $\cE_v$ denote the set of all edges which have $v$ as one of its
endpoints. We assume that the valency $\val (v)$ of any vertex $v$ is finite and
positive. In particular, vertices without incident edges are not allowed.

An $n$-tuple $p = (v_i)_{l = 0}^n$ of vertices is called a path in $(\cV, \,
\cE_{ud})$ if $\{v_i, \, v_{i+1}\} \in \cE_{ud}$ for every $i = 0, \, \ldots, \,
n-1$. In this case, we call $p$ a path joining $v_0$ with $v_n$. An undirected
graph is connected if any two of its vertices are connected by a path, and a
digraph is connected if the related undirected graph is connected (thus,
connectedness is defined independently of orientation).

A function $l : \cE \to (0, \, \infty)$ is called a {\em length function}, and
$l(e)$ is called the length of the edge $e$. Each triple $\Gamma_{metr} =
(\cV, \, \cE, \, l)$ determines a metric graph by identifying the edge $e$
with the line segment $[0, \, l(e)]$. The formal definition is as follows (the
construction we use is quite similar to the definition in \cite{LSS}). Let
\[
\Gamma^\sim_{metr} := \{ (e, \, x) \in \cE \times  [0, \, \infty): e \in \cE, \,
x \in [0, \, l(e)] \}
\]
and consider the function $\Pi^\sim : \Gamma^\sim_{metr} \to \cV \cup \cE$
defined by
\[
\Pi^\sim (e, \, x) = \Pi^\sim ((v, \, w), \, x) :=
\left\{ \begin{array}{ll}
e & \mbox{if} \; x \in (0, \, l(e)), \\
v & \mbox{if} \; x=0, \\
w & \mbox{if} \; x=l(e).
\end{array}
\right.
\]
Two points $(e, \, x), \, (f, \, y) \in \Gamma^\sim_{metr}$ are said to be
equivalent if $\Pi^\sim (e, \, x)$ and $\Pi^\sim (f, \, y)$ are in $\cV$
and if $\Pi^\sim (e, \, x) = \Pi^\sim (f, \, y)$. This defines an equivalence
relation on $\Gamma^\sim_{metr}$ which we denote by $\sim$. The equivalence
class of $(e, \, x)$ with respect to $\sim$ is denoted by $(e, \, x)^\sim$.
Clearly, the equivalence class of each point $(e, \, x)$ with $x \neq 0$ and $x
\neq l(e)$ is a singleton, whereas points $(e, \, x)$ with $x = 0$ or $x
= l(e)$ are identified if they belong to the ``same'' vertex of $\Gamma_{comb}$.

The set $\Gamma_{metr} := \Gamma^\sim_{metr}/\sim$ of all equivalence classes
is called the {\em metric graph} associated with $(\cV, \, \cE, \, l)$. Since
$\Pi^\sim (e, \, x)$ only depends on the equivalence class of $(e, \, x)$,
we can define the quotient map
\[
\Pi : \Gamma_{metr} \to \cV \cup \cE, \;  (e, \, x)^\sim \mapsto \Pi^\sim (e,
\, x).
\]
The elements of $\Pi^{-1} (\cV)$ and $\Pi^{-1} (\cE)$ are called the vertices
and the open edges of $\Gamma_{metr}$, whereas the unions $\Pi^{-1} (v \cup (v,
\, w) \cup w)$ for an edge $e = (v, \, w) \in \cE$ are called the closed edges
of $\Gamma_{metr}$. Thus, a closed edge is the union of an open edge with its
end points.

There is a natural topology on a metric graph which is defined as follows.
Provide $\cE$ with the discrete topology and $[0, \, \infty)$ with the standard
(Euclidean) topology and consider on $\Gamma^\sim_{metr}$ the restriction of
the product topology on $\cE \times [0, \, \infty)$. Then the topology on the
metric graph $\Gamma_{metr}$ is the quotient of the topology on
$\Gamma^\sim_{metr}$ by the relation $\sim$.

Moreover, this topology is induced by a metric (whence the notion {\em metric
graph}). Given an edge $e = (v, \, w) \in \cE$ and a point $(e, \, x)^\sim \in
\Gamma_{metr}$ with $x \neq 0$ and $x \neq l(e)$, we call $[v, \, x] := \{ (e,
\, y) : 0 \le y \le x\}$ and $[x, \, w] := \{ (e, \, y) : x \le y \le l(e)\}$
the segments joining $(e, x)$ with the end points of the edge $e$. With the
segments $[v, \, x]$ and $[x, \, w]$ we associate the lengths $x$ and $l(e) -
x$, respectively.

Let $(e, \, x)^\sim, \, (f, \, y)^\sim \in \Gamma_{metr}$. By a path between
$(e, \, x)^\sim$ and $(f, \, y)^\sim$ we mean an $n$-tuple $p = (v_i)_{l = 0}^n$
of vertices in $\cV$ such that $v_0$ is an endpoint of the edge $e$ and $v_n$
is an endpoint of $f$ and $v_0$ and $v_n$ are connected by a path in
$\Gamma_{comb}$ as above. The length of this path is defined as the sum of the
lengths of the segments joining $(e, \, x)^\sim$ and $(f, \, y)^\sim$ with the
corresponding endpoints $v_0$ and $v_n$, respectively, plus the sum
$\sum_{i=0}^{n-1} l(e_i)$ where $e_i$ is the edge $(v_i, \, v_{i+1}) \in \cE$.
If $e=f$ we also consider the segment $[x, \, y] := \{ (e, \, z) : x \le z \le
y\}$ as a path of length $y-x$ between $(e, \, x)^\sim$ and $(e, \, y)^\sim$.
The distance of $(e, \, x)^\sim$ and $(f, \, y)^\sim$ is then the infimum of the
lengths of all paths joining these points. Note that the distance of any two
different points is positive since every vertex has finite valency.

Since the edges of $\Gamma_{comb}$ and the open (resp. closed) edges of
$\Gamma_{metr}$ are in a one-to-one correspondence, we often use the same
notation $e$ both for an edge in $\cE$ and for the corresponding open (resp.
closed) edge of $\Gamma_{metr}$. Moreover, we identify the open (resp. closed)
edge of $\Gamma_{metr}$ which corresponds to $e \in \cE$ with the open (resp.
closed) interval $(0, \, l(e))$ (resp. $[0, \, l(e)]$). Accordingly, we usually
identify the point $(e, \, x)^\sim$ with $x$ and write $x \in e$ in order to
indicate that $x$ is considered as a point of the metric graph. Finally, we
provide $\Gamma_{metr}$ by the measure which is induced by the one-dimensional
Lebesgue measure on each edge.

In what follows, we let $\Gamma = (\cV, \, \cE, \, l)$ be a metric graph which
is periodic with respect to $\sZ^n$ (or $\sZ^n$-periodic for short) in the
following sense: The group $\sZ^n$ acts freely on $\Gamma$, i.e., there is a
mapping
\[
\Gamma \times \sZ^n \to \Gamma, \; (x, \, g) \mapsto x + g
\]
such that $x + 0 = x$ and $x + (g_1 + g_2) = (x + g_1) + g_2$ for all
$g_1, \, g_2 \in \sZ^n$ and $x \in \Gamma$, and if $x = x + g$ for some $x \in
\Gamma$ and $g \in \sZ^n$, then $g=0$. Moreover, we assume that every mapping
$x \to x+g$ sends vertices to vertices and edges to edges such that $v+g$
and $w+g$ are the endpoints of the image $e+g$ of the edge $e$ with endpoints
$v$ and $w$ and that the lengths of $e$ and $e+g$ are equal (thus, the
length function $l$ is $\sZ^n$-periodic). Then both the valency and the metric
on $\Gamma$ are invariant with respect to the action of $\sZ^n$, that is
$\val (v+g) = \val (v)$ for every vertex $v \in \cV$ and $\rho (x + g, \, y + g)
= \rho (x, \, y)$ for arbitrary points $x, \, y \in \Gamma$ and every $g \in
\sZ^n$. Moreover, we assume that
\begin{equation}  \label{1.1''}
\lim_{\sZ^n \ni g \to \infty} \rho (x, \, y+g) = \infty \quad \mbox{for} \; x,
\, y \in \Gamma.
\end{equation}
If these conditions are satisfied, we call $\Gamma$ a $\sZ^n$-periodic metric
graph.

In what follows we also suppose that the fundamental domain $\Gamma_0 :=
\Gamma/\sZ^n$ of $\Gamma$ with respect to the action of $\sZ^n$ is compact in
the corresponding quotient topology (thus, the action of $\sZ^n$ on $\Gamma$ is
co-compact). Note that this property implies that $\Gamma_0$ contains only a
finite number of vertices of $\Gamma$. For $g \in \sZ^n$, we set $\Gamma_g :=
\{y \in \Gamma : y \in \Gamma_0 + g\}$. Then $\Gamma_{g_1} \cap \Gamma_{g_2} =
\emptyset$ if $g_1 \neq g_2$ and $\Gamma = \cup_{g \in \sZ^n} \Gamma_g$.

A function $w : \Gamma \to [0, \, \infty)$ is called a weight if $w$ is
continuous on $\Gamma \setminus \cV$ and $w^{-1} (\{0, \, \infty\}) \subset
\cV$. For $p \in [1, \, \infty]$, we let $L_w^p (\Gamma)$ denote the space of
all measurable functions on $\Gamma$ such that
\[
\|u\|_{L_w^p (\Gamma)}^p := \int_{\Gamma} |w(x)u(x)|^p dx
\]
in case $p < \infty$ and
\[
\|u\|_{L_w^\infty (\Gamma)} := \mbox{esssup}_{x \in \Gamma} |w(x)u(x)|
\]
if $p = \infty$ are finite, respectively. We write $L^p (\Gamma)$ instead
of $L_w^p (\Gamma)$ if $w$ is identically 1.

In what follows we suppose that the weight is periodic with respect to $\sZ^n$,
that is $w \circ g = w$ for all $g \in \sZ^n$ (where we identify $g$ with the
mapping $x \mapsto x+g$). Under this condition, the spaces $L_w^p (\Gamma)$ are
invariant with respect to the action of $\sZ^n$, i.e., $\|u \circ g\|_{L_w^p
(\Gamma)} = \|u\|_{L_w^p (\Gamma)}$ for every $g \in \sZ^n$. For $u \in L_w^p
(\Gamma)$ and $g \in \sZ^n$, set $u_g := u|_{\Gamma_g}$. Then
\[
\|u\|_{L_w^p (\Gamma)}^p = \sum_{g \in \sZ^n} \|u_g\|_{L_w^p(\Gamma_g)}^p
\]
for $p \in [1, \, \infty)$ and
\[
\|u\|_{L_w^\infty (\Gamma)} = \sup_{g \in \sZ^n} \|u_g\|_{L_w^\infty
(\Gamma_g)}.
\]
For $g \in \sZ^n$, we let $V_g$ denote the operator of shift by $g$ which acts
on functions in $L_w^p (\Gamma)$ as $(V_g u)(x) := u(x-g)$. Since $w$ is
$\sZ^n$-periodic, the operators $V_g$ are isometries on $L_w^p (\Gamma)$, and
$V_g^{-1} = V_{-g}$.

Every weighted Lebesgue space $L_w^p (\Gamma)$ over a $\sZ^n$-periodic metric
graph is naturally isomorphic to an $l^p$-space of vector-valued
$\sZ^n$-sequences as follows. Let $w_0 := w|_{\Gamma_0}$ and $X := L_{w_0}^p
(\Gamma_0)$, and consider the operator
\[
U : L_w^p (\Gamma) \to l^p(\sZ^n, \, X), \quad (Uu)(g) = (V_g u)|_{\Gamma_0}.
\]
The operator $U$ is an isometry with inverse $U^{-1} : l^p (\sZ^n, \, X) \to
L_w^p (\Gamma)$ acting as
\[
U^{-1} f = \sum_{g \in \sZ^n} \chi_g V_g f_g \chi_0 I
\]
where  $\chi_0$ refers to the characteristic function of $\Gamma_0$ and $\chi_g
:= V_g^{-1} \chi_0$.

Let $A \in \cB (L_w^p (\Gamma))$. Then $\tilde{A} := UAU^{-1} \in \cB(l^p
(\sZ^n, \, X))$ has the matrix representation
\[
(\tilde{A}\psi)_\alpha = \sum_{\beta \in \sZ^n} \tilde{A}_{\alpha \beta}
\tau_\beta \psi_\alpha, \quad \alpha \in \sZ^n,
\]
where
\[
\tilde{A}_{\alpha \beta} = \chi_0 V_{-\alpha} A V_{\alpha -\beta} \chi_0 I
= V_{-\alpha} \chi_\alpha A \chi_{\alpha -\beta} V_{\alpha -\beta}.
\]
We say that the operator $A \in \cB(L_w^p (\Gamma))$ belongs to the Wiener
algebra $W_w (\Gamma)$ if
\[
\|A\|_{W_w (\Gamma)} := \sup_{\alpha \in \sZ^n} \sum_{\beta \in \sZ^n}
\|\chi_\alpha A \chi_{\alpha - \beta} I\|_{\cB (L_w^p (\Gamma))} < \infty.
\]
Then the mapping $A \mapsto U A U^{-1}$ is an isometric isomorphism between
the Wiener algebras $W_w (\Gamma)$ and $W(\sZ^n, \, X)$. Because operators in
$W(\sZ^n, \, X)$ are bounded on $l^p (\sZ^n, \, X)$, operators in $W_w (\Gamma)$
are bounded on $L_w^p (\Gamma)$.
\subsection{Simonenko's local principle}
We will base our study of pseudodifferential operators on periodic graphs on
Simonenko's local principle which we recall here briefly from \cite{Sim} in a
form which is convenient for our purposes; see also Section 2.5 in \cite{RSS}.

We start with some definitions. Let $\dot{\Gamma}$ denote the one-point
compactification of the periodic metric graph $\Gamma$. For a measurable subset
$F$ of $\dot{\Gamma}$, let $\chi_F$ denote its characteristic function. An
operator $A \in \cB(L_w^p(\Gamma))$ is said to be {\em of local type} on
$\dot{\Gamma}$ if, for arbitrary open sets $F_1, \, F_2 \subset \dot{\Gamma}$
with disjoint closures, the operator $\chi_{F_1} A \chi_{F_2} I$ is compact on
$L_w^p (\Gamma)$. Further, the operator $A \in \cB(L_w^p(\Gamma))$ is called a
{\em locally Fredholm operator at} $x \in \dot{\Gamma}$ if there are
an open neighborhood $F$ of $x$ and operators $L, \, R \in \cB(L_w^p(\Gamma))$
such that
\[
L A \chi_F I - \chi_F I, \quad \chi_F A R - \chi_F I \in \cK(L_w^p (\Gamma)).
\]
Finally, we say that $A \in \cB(L_w^p (\Gamma))$ is {\em locally invertible at
infinity} if there are a positive constant $r$ and operators $L, \, R \in
\cB(L_w^p (\Gamma))$ such that
\[
L A \chi_{B_r^\prime} I = \chi_{B_r^\prime} I \quad \mbox{and} \quad
\chi_{B_r^\prime} A R = \chi_{B_r^\prime} I,
\]
where $B_r^\prime := \{x \in \dot{\Gamma}: \rho(x, \, x_0) > r\}$ for a certain
fixed point $x_0 \in \Gamma$ (one easily checks that the property of being
locally invertible at infinity does not depend on the choice of $x_0$). For $p
\in (1, \, \infty)$, it is also easy to see that an operator $A \in
\cB (L_w^p (\Gamma))$ is locally Fredholm at infinity if and only if it is
locally invertible at infinity. The following is one of the main results of
\cite{Sim}.
\begin{proposition} \label{p1} 
Let $A \in \cB(L_w^p (\Gamma))$ be an operator of local type on $\dot{\Gamma}$. 
Then $A$ is a Fredholm operator if and only if $A$ is a locally Fredholm 
operator at every point $x \in \dot{\Gamma}$.
\end{proposition}
\begin{definition} \label{d4}
Let $A \in \cB(L_w^p (\Gamma))$ and $h \in \cH$. An operator $A^h \in \cB(L_w^p
(\Gamma))$ is called a {\em limit operator} of $A$ defined by $h$ if, for every
compact subset $M$ of $\Gamma$,
\[
\lim_{m \to \infty} \|(V_{h(m)}^{-1} A V_{h(m)} - A^h) \chi_M
I\|_{\cB(L_w^p (\Gamma))} = 0
\]
and
\[
\lim_{m \to \infty} \|\chi_M (V_{h(m)}^{-1} A V_{h(m)} - A^h)\|
_{\cB(L_w^p (\Gamma))} = 0.
\]
\end{definition}
If $p \in (1, \, \infty)$, one can show that $A^h$ is the limit operator of $A$
with respect to $h$ if and only if
\[
V_{h(m)}^{-1} A V_{h(m)} \to A^h \quad \mbox{and} \quad
V_{h(m)}^{-1} A^* V_{h(m)} \to (A^h)^*
\]
strongly as $m \to \infty$. We let $Lim(A)$ denote the set of all limit
operators of $A$, and we call an operator $A$ {\em rich} if every sequence
$h \in \cH$ has a subsequence $g$ such that the limit operator $A^g$ of
$A$ with respect to $g$ exists.
\begin{proposition} \label{p2}
Let $A \in W_w (\Gamma)$ be a rich operator. Then the operator $A$, considered
as an element of $\cB (L_w^p (\Gamma))$, is locally invertible at infinity if
and only if all limit operators of $A$ are invertible on $L_w^p (\Gamma)$.
\end{proposition}
\begin{proof}
Let again $X := L_{w_0}^p (\Gamma)$ and consider the operator $\tilde{A} =
UAU^{-1} \in W(\sZ^n, \, X)$. It is easy to see that if the limit operator of
$A$ with respect to $h \in \cH$ exists, then the limit operator of $\tilde{A}$
with respect to $h$ exists, too, and $\tilde{A}^h = U A^h U^{-1}$. In
particular, $\tilde{A}^h$ is invertible on $l^p (\sZ^n, \, X)$ if and only
if $A^h$ is invertible on $L_w^p (\Gamma)$. Moreover, the operator $A :
L_w^p (\Gamma) \to L_w^p (\Gamma)$ is locally invertible at infinity if and only
if the operator $\tilde{A} : l^p (\sZ^n, \, X) \to l^p (\sZ^n, \, X)$ has
this property. Hence, the assertion follows from Theorem 2.5.7 in \cite{RRS}.
\end{proof}

The following theorem is then an immediate consequence of Propositions \ref{p1}
and \ref{p2}.
\begin{theorem} \label{t1}
Let $A \in W_w (\Gamma)$ be both rich and of local type on $L_w^p
(\dot{\Gamma})$. Then $A$ is a Fredholm operator on $L_w^p (\Gamma)$ if and only
if $A$ is a locally Fredholm operator at every point $x \in \Gamma$ and if all
limit operators of $A$ are invertible on $L_w^p (\Gamma)$.
\end{theorem}
\subsection{The Fredholm property of pseudodifferential operators on periodic
graphs}
Let $\Gamma $ be a $\sZ^n$-periodic metric graph and $\cV$ and $\cE$ the sets of
its vertices and edges, respectively. As we agreed above, we identify every
edge $e = (v, \, w)$ with the directed segment $[0, \, l(e)]$ with endpoints
$v, \, w$ and consider the distance between $x \in e$ and $v$ as the local
coordinate of the point $x$.

Next we describe a class of operators on $L_w^p (\Gamma)$ for which we will
derive a Fredholm criterion below. First we have to specify the weight
function. Let $F_v \subset \Gamma$ be a sufficiently small neighborhood of a
vertex $v \in \cV$. Then we can think of this neighborhood as the union $F_v =
\cup_{j=1}^{\val (v)} \gamma_j^v$ where the $\gamma_j^v$ are
segments of edges incident to $v$. We suppose that $F_v$ is such that all
segments $\gamma_j^v$ have the same length, thus they can be identified with a
common interval $[0, \, \varepsilon_v]$, that $w|_{\gamma_j^v} =: w_v$ is
independent of $j$, and that $w_v = e^{\sigma_v} \in \cR(\cI_v)$ for some open
interval $\cI_v$. The operators $A$ under consideration are supposed to satisfy
the following conditions:
\begin{enumerate}
\item[A1]
For every open edge $e$ of $\Gamma$, there is a symbol $a_e \in S^0(\sR)$ such
that $\varphi A \psi I = \varphi Op(a_e) \psi I$ for all functions $\varphi,
\, \psi \in C_0^\infty (e)$.
\item[A2]
Let $v \in \cV$ and $F_v$ a small neighborhood of $v$, as specified above. Then
there is a symbol $a_v \in \cE(\val (v), \, \cI_v)$ such that $\varphi A \psi I
= \varphi op(a) \psi I$ for all functions $\varphi, \, \psi \in C_0^\infty
(F_v)$.
\item[A3]
There is a function $f \in l^1 (\sZ^n)$ such that
\[
\|\chi_\alpha A \chi_\beta I\|_{L_w^p(\Gamma)} \le f(\alpha - \beta) \quad
\mbox{for all} \; \alpha, \, \beta \in \sZ^n.
\]
\item[A4]
The operator $A$ is of local type on $\dot{\Gamma}$.
\item[A5]
The operator $A$ is rich.
\end{enumerate}
Assumptions A1 and A2 guarantee that $A$ behaves as a (Fourier) $\psi$do along
every edge and as a Mellin $\psi$do in a neighborhood of every vertex.
Assumption A3 implies that $A \in W_w(\Gamma )$ and, hence, $A$ is bounded
on $L_w^p (\Gamma)$.
\begin{theorem} \label{te1}
Let $A$ satisfy assumptions {\rm A1-A5}. Then $A$ is a Fredholm operator on
$L_w^p(\Gamma)$ if and only if the following conditions hold: \\
$(a)$ for every open edge $e$ of $\Gamma$ and every $x \in e$,
\[
\liminf_{\xi \to \infty} |a_e (x, \, \xi)| > 0;
\]
$(b)$ for every vertex $v \in \cV$,
\[
\liminf_{r \to +0} \inf_{\lambda \in \sR} |\det a_v (r, \, \lambda
 + i \varkappa_{\sigma_v} (r))| > 0;
\]
$(c)$ all limit operators of $A$ are invertible on $L_w^p (\Gamma)$.
\end{theorem}
\begin{proof}
It follows from Propositions \ref{prop2} and \ref{prop3} that condition $(a)$
is necessary and sufficient for the local Fredholmness of $A$ at the every point
$x \in \Gamma \setminus \cV$, whereas condition (ii) is necessary and
sufficient for the local Fredholmness at the every point $v \in \cV$, and
condition (iii) is necessary and sufficient for the local Fredholmness
(= local invertibility) at the point $\infty$. The assertion follows then from
Theorem \ref{t1}.
\end{proof}
\begin{corollary} \label{co1}
Let $A$ satisfy assumptions {\rm A1-A5} and conditions $(a)$ and $(b)$ in
Theorem $\ref{te1}$. Then
\[
ess\,sp_{L_w^p (\Gamma)} \, A = \cup_{A^g \in Lim(A)} sp_{L_w^p (\Gamma)}
A^g.
\]
\end{corollary}
\section{Applications}
\subsection{Singular integral operators}
In this section, we let $\Gamma$ be metric graph which is embedded into
$\sR^2$. Thus, the vertices of $\Gamma$ {\em are} points, and the edges of
$\Gamma$ {\em are} line segments in the plane. We suppose that $\Gamma$
is $\sZ^n$-periodic with $n \in \{1, \, 2\}$.

In the following definition of singular integral operators we have to divide by
$y-x = (y_1 - x_1, \, y_2 - x_2) \in \sR^2$. This division is understood in the
complex sense, i.e., as a (complex) division by $(y_1 - x_1) + i (y_2 - x_2)$.
Further we say that a function $\phi : \Gamma \times \sR^2 \to \sC$ belongs
to $C^\infty (\Gamma \times \sR^2)$ if, for every $\alpha \in \sN_0$ and every
multi-index $\beta \in \sN_0^2$, the partial derivatives $\partial_x^\alpha
\partial_z^\beta (\phi|_{e \times \sR^2})$ exist on every set $e \times
\sR^2$ where $e$ is an open edge of $\Gamma$, and if these partial derivatives
can be continuously extended to the closure $\bar{e} \times \sR^2$ of $e
\times \sR^2$.

We consider singular integral operators of the form
\begin{equation} \label{s0}
(\cS_{\Gamma, \phi} u)(x) = \frac{1}{\pi i} \int_\Gamma \frac{\phi(x, \, x-y)
u(y)}{y-x} dy, \quad x \in \Gamma,
\end{equation}
where the integral is understood in the sense of the Cauchy principal value,
\[
\lim_{\varepsilon \to 0} \int_{\Gamma \cap \{y : |y-x| < \varepsilon\}}
\frac{\phi(x, \, x-y) u(y)}{y-x} dy,
\]
and the function $\phi \in C^\infty (\Gamma \times \sR^2)$ owns the property
that, for all $\alpha, \, N \in \sN_0$ and $\beta \in \sN_0^2$, there is a
constant $C_{\alpha \beta N}$ such that
\begin{equation}  \label{s1}
|\partial_x^\alpha \partial_z^\beta \phi (x, \, z)| \le C_{\alpha \beta N}
(1 + |z|)^{-N} \quad \mbox{for all} \; (x, \, z) \in \Gamma \times \sR^2.
\end{equation}
In order to show that $\cS_{\Gamma, \phi}$ is a pseudodifferential operator on
$\Gamma$ in the sense of our previous definitions, we consider its restriction
$\cS_{\Gamma, \phi, e}$ to a single edge $e$ of $\Gamma$. Identifying this edge
with the interval $\{(t, \, 0) \in \sR^2 : t \in (0, \, l)\}$, we can write this
restricted operator in the form
\[
(\cS_{\Gamma, \phi, e}v)(t) = \frac{1}{\pi i} \int_0^l \frac{\psi(t, \, t
-\tau) v(\tau)}{\tau - t}  d\tau, \quad t \in (0, \, l),
\]
where $\psi(t, \, \tau) := \phi((t, \, 0), \, \tau, \, 0)$ with $t, \, \tau \in
(0,\, l)$. Hence, $\cS_{\Gamma, \phi, e}$ can be identified with the
restriction onto the interval $(0, \, l) \subset \sR$ of the operator
\[
(S_{\sR, \phi} f)(x) = \frac{1}{\pi i} \int_{\sR}
\frac{\psi(x, \, x-y) f(y)}{y-x} dy, \quad x \in \sR,
\]
We have to show that this operator is a (Fourier) $\psi$do in the class
$OPS^0$ with main symbol $sgn$. This fact will follow from Proposition \ref{pr3}
below. We call the function
\[
\sigma _{S_{\sR, \phi}}(x, \, \xi) := v.p. \frac{1}{\pi i} \int_{-\infty}
^\infty \frac{\phi(x, \, z)}{z}e^{iz \xi} dz
\]
on $\sR \times \sR$ the symbol of the operator $S_{\sR, \phi}$. The main
properties of this function are summarized in the following proposition.
\begin{proposition} \label{pr3}
The function $\sigma _{S_{\sR, \phi}}$ satisfies $(\ref{s1})$ in place of
$\phi$, and
\begin{equation} \label{s2}
\sigma_{S_{\sR, \phi}} (x, \, \xi) = \phi (x, \, 0) {\rm sgn \,} \xi + q_\phi
(x, \, \xi)
\end{equation}
where the function $q_\phi$ is such that $\partial_x^\alpha \partial_\xi^\beta
q_\phi (x, \, \xi) = O(|\xi|^{-N})$ for all $\alpha, \, \beta, \, N \in
\sN_0$ uniformly with respect to $x \in \sR$.
\end{proposition}
\begin{proof}
The estimates (\ref{s1}) follow from the identity
\[
\frac{\partial \sigma_{S_{\sR, \phi}} (x, \, \xi)}{\partial \xi} =
\frac{1}{\pi} \int_{-\infty }^\infty \phi(x, \, z) e^{iz \xi} dz
\]
by integrating by parts, whereas (\ref{s2}) is a consequence of the
asymptotic behavior of singular integral operators as $\xi \to \infty$; see,
for instance, page 112 in \cite{Fed}.
\end{proof}

\begin{proposition} \label{p3}
Assume that, for every $v \in \cV$, the weight $w$ is such that
$w_v = \exp \sigma_v \in \cR(\cI)$ where $\cI = (-\frac{1}{p}, \,
1 - \frac{1}{p})$ with $p \in (1, \, \infty)$. Then $\cS_{\Gamma, \phi} \in
W_w (\Gamma)$ and, consequently, $\cS_{\Gamma, \phi}$ is a bounded operator on
$L_w^p (\Gamma)$.
\end{proposition}
\begin{proof}
First we prove that the operator $\cS_{\Gamma_0, \phi}$ is bounded on
$L_w^p (\Gamma_0)$. Indeed, $\cS_{\Gamma_0, \phi} = S_{\Gamma_0} +
K_{\Gamma_0, \phi}$, where
\[
(S_{\Gamma_0} u)(x) = \frac{\phi (x, \, 0)}{\pi i} \int_{\Gamma_0} \frac{u(y)}
{y-x} dy, \quad x \in \Gamma_0,
\]
is (a multiple of) the standard singular integral operator and
\[
(K_{\Gamma_0, \phi} u)(x) = \int_{\Gamma_0} k(x, \, y) u(y) dy, \quad x \in
\Gamma_0,
\]
where
\[
k(x, \, y) =  \frac{1}{\pi i} \frac{(\phi(x, \, x-y) - \phi(x, \, 0))}{y-x}
\]
Note that the condition imposed on the weight implies the boundedness of
$S_{\Gamma_0}$ on $L_w^p (\Gamma_0)$ (see, for instance, \cite{BK},
Section 4.5). Moreover, the hypotheses of the proposition guarantee that
$w \in L^p (\Gamma_0)$ and $w^{-1} \in L^q (\Gamma_0)$ where $\frac{1}{p} +
\frac{1}{q} = 1$. Since $k \in L^\infty (\Gamma_0 \times \Gamma_0)$, the
operator $K_{\Gamma_0, \phi}$ is bounded from $L^1 (\Gamma_0)$ to $L^\infty
(\Gamma_0)$. Hence, $w K_{\Gamma_0, \phi} w^{-1} I$ is bounded on $L^p
(\Gamma_0)$, whence the boundedness of $K_{\Gamma_0, \phi}$ on $L_w^p
(\Gamma_0)$.

Next we prove that $\cS_{\Gamma, \phi} \in W_w (\Gamma)$. Let $\alpha, \,
\beta \in \sZ^n$. It is easy to see that the operators defined by
\begin{eqnarray*}
(\chi_0 V_{-\alpha} \cS_{\Gamma, \phi} V_{\alpha - \beta} \chi_0 v) (x)
& = & \frac{1}{\pi i} \int_{\Gamma_0} \frac{\phi(x, \, x-y+\beta)}{y-x+\beta}
v(y) dy \\
& =: & (A_\beta v) (x), \quad x \in \Gamma_0,
\end{eqnarray*}
satisfy the estimate
\begin{equation}  \label{1}
\|A_\beta\|_{\cL (L^1 (\Gamma_0), L^\infty(\Gamma_0))} \le
\sup_{(x,y) \in \Gamma_0 \times \Gamma_0} \frac{|\phi (x, \, x-y+\beta)|}
{|y-x+\beta|}
\end{equation}
for all $\beta$ sufficiently large. Moreover, there is a constant $M>0$
such that
\[
\frac{1}{|y - x + \beta|} \le M
\]
for all $x, \, y \in \Gamma_0$ and all sufficiently large $\beta \in \sZ^n$.
Hence, estimate (\ref{1}) implies
\[
\|A_\beta\|_{\cL(L^1(\Gamma_0), L^\infty(\Gamma_0))} \le C_N |\beta|^{-N}
\]
for every $N \in \sN$ and $\beta \in \sZ^n$ large enough, which finally
yields
\begin{eqnarray*}
\|A_\beta\|_{\cL(L_w^p(\Gamma_0))}
& = & \|w A_\beta w^{-1}I\|_{\cL(L^p(\Gamma_0))} \\
& \le & \|w\|_{L^p (\Gamma_0)} \|A_\beta\|_{\cL(L^1(\Gamma_0), L^\infty
(\Gamma_0))} \|w^{-1}\|_{L^q(\Gamma_0)} \\
& \le & C_N^\prime |\beta|^{-N}
\end{eqnarray*}
for all sufficiently large $\beta$, where we wrote $C_N^\prime :=
\|w\|_{L^p(\Gamma_0)} \|w^{-1}\|_{L^q(\Gamma_0)} C_N$. Thus,
$\cS_{\Gamma, \phi} \in W_w(\Gamma)$.
\end{proof}
\subsection{The Fredholm property of operators $aI + b \cS_{\Gamma, \phi}$}
We suppose that the coefficients $a, \, b$ of the operator $A := aI
+ b \cS_{\Gamma, \phi}$ are bounded piecewise continuous functions on
$\Gamma$ which have only discontinuities of the first kind (i.e., jumps) and
which are smooth on $\Gamma \setminus \cV$. More precisely, for $\omega \in
\cV$, let $F_\omega$ be a sufficiently small neighborhood of $\omega$ such that
\[
\Gamma \cap F_\omega = \cup_{j=1}^{\val (\omega )} \gamma_j^\omega
\]
where the $\gamma_j^\omega$ are rays of the same length incident to the vertex
$\omega$. On $F_\omega$, we introduce a local system of coordinates such that
\[
\gamma_j^\omega  = \{z = \omega + r e^{i \theta_j^\omega } : r \in (0,
\, \varepsilon) \}
\]
where $0 \le \theta_1^\omega < \theta_2^\omega < \ldots < \theta_{\val(\omega)}
^\omega < 2 \pi$, and we put
\[
f_j^\omega (r) := f(\omega +r e^{i \theta_j^\omega}), \quad r \in (0,
\varepsilon)
\]
and $f^\omega := \diag (f_1^\omega, \, \ldots, \, f_{\val (\omega)}^\omega)$
for every function $f$ on $\Gamma$. We say that a function $f \in L^\infty
(\Gamma)$ belongs to the class $PC^\infty (\Gamma)$ if $f \in C^\infty (\Gamma
\setminus \cV)$ and if the one-sided limits
\[
\lim_{r \to 0} f_j^\omega (r) =: f_j(\omega)
\]
exist for every vertex $\omega$ and every $j \in \{1, \, \ldots, \, \val
(\omega)\}$. In this case we write
\begin{equation} \label{eneu1}
\tilde{f}(\omega) := \diag (f_1(\omega ), \, \ldots, \, f_{\val (\omega)}).
\end{equation}
For $\omega \in \cV$, we put $\varepsilon_k := 1$ if $\omega$ is the starting
point of the oriented edge $e_k^\omega \supset \gamma_k^\omega$ and
$\varepsilon_k := -1$ if $\omega$ is the terminating point of the edge
$e_k^\omega$. Further we define functions $\nu : [0, \, 2 \pi) \times (\sC
\setminus i \sZ) \to \sC$ by
\[
\nu(\delta, \, \zeta) := \left\{
\begin{array}{ll}
\coth (\pi \zeta) & \mbox{if} \; \delta = 0, \\
\frac{e^{(\pi -\delta) \zeta}}{\sinh (\pi \zeta)} & \mbox{if} \; \delta \in
(0, \, 2 \pi)
\end{array}
\right.
\]
and $s_{jk}^\omega : \sC \setminus i\sZ \to \sC$ by
\begin{eqnarray*}
s_{jk}^\omega (\zeta) := \varepsilon_k \left\{
\begin{array}{ll}
\nu(2 \pi + \theta_j^\omega - \theta_k^\omega, \, \zeta) & \mbox{if} \; j < k,
\\
\nu(0, \, \zeta) & \mbox{if} \; j = k, \\
\nu(\theta_j^\omega - \theta_k^\omega, \, \zeta) & \mbox{if} \;
j > k
\end{array}
\right.
\end{eqnarray*}
and we set, for $(r, \, \lambda) \in (0, \, \varepsilon) \times \sR$,
\begin{equation} \label{eneu2}
\left( \hat{\sigma}_\omega (\cS_\Gamma)c \right)(r, \, \lambda) := \left(
s_{jk}^\omega \left( \lambda + i \left( \frac{1}{p} + \varkappa_{w_\omega} (r)
\right) \right) \right) _{j, k=1}^{\val (\omega)}
\end{equation}
where $\varkappa_{w_\omega} (r) := r \frac{dv_\omega (r)}{dr}$ for $r \in
(0, \, \varepsilon)$.

We consider the operator $A = aI + b\cS_{\Gamma, \phi}$ with coefficients
$a, \, b \in PC^\infty (\Gamma)$. Our goal is to define the symbol of $A$ at
every point of $\Gamma$ and at the infinitely distant point $\infty$ in such a
way that the invertibility of the symbol corresponds to the Fredholm property
of $A$.

If $\omega \in \cV$, then the symbol at $\omega$ is the function
\[
\sigma_A^\omega (r, \, \lambda) := \tilde{a}(\omega) + \tilde{\phi}(\omega, \,
0) \tilde{b}(\omega) \hat{\sigma}_\omega (\cS_\Gamma)(r, \, \lambda),
\quad (r, \, \lambda) \in (0, \, \varepsilon) \times \sR,
\]
where $\tilde{a}(\omega)$, $\tilde{b}(\omega)$ and $\hat{\sigma}_\omega
(\cS_\Gamma)$ are given by (\ref{eneu1}) and (\ref{eneu2}), respectively.

If $x \in \Gamma \setminus \cV$, then the symbol at $x$ is given by
\[
\sigma _{A}^{x}(\xi) := a(x) + b(x) \phi(x, \, 0) \mbox{sgn} \, \xi, \quad \xi
\in \sR.
\]
The description of the symbol at $\infty$ is more involved. We will
employ the limit operators of $A$ for this goal. First we define the symbol
of the operator of multiplication by a function $f \in PC^\infty (\Gamma)$.
We write the graph $\Gamma$ as a countable union $\cup _{j \in \sN} e_j$
of edges. It is then a consequence of the Arzela-Ascoli theorem and a
standard diagonal argument that every sequence $h \in \cH$ has a subsequence $g$
such that
\[
(f|_{e_j}) (x + g(m)) \to f_j^g (x)
\]
uniformly on $e_j$ for every $j \in \sN$. With the so-defined family
$\{f_j^g\}_{n \in \sN}$ of functions on the edges of $\Gamma$, we associate a
function $f^g$ on all of $\Gamma$ in the natural way. The function $f^g$ has
the property that, for every compact subset $K$ of $\Gamma $,
\[
\lim_{m \to \infty} \sup_{x \in K} |f(x + g(m)) - f^g(x)| = 0.
\]
Hence $f^g I$ is the limit operator of $fI$ with respect to $g$. Moreover, we
obtained that $fI$ is a rich operator.

The limit operators of $fI$ are of a particularly simple form if $f$ is slowly
oscillating at infinity. This class of functions is defined as follows. Let
$f \in L^\infty(\Gamma)$. We represent $f$ in the form
\[
f(x) = f(y + \alpha) =: f_y(\alpha), \quad y \in \Gamma_0, \; \alpha \in \sZ^n,
\]
with $\Gamma_0$ again referring to the fundamental domain of $\Gamma$, and we
call $f$ {\em slowly oscillating at infinity} if
\begin{equation} \label{e2}
\lim_{\alpha \to \infty} \sup_{y \in \Gamma_0} |f_y (\beta + \alpha) -
f_y(\alpha)| = 0 \quad \mbox{for every} \; \beta \in \sZ^n.
\end{equation}
We denote the class of all functions which are slowly oscillating at infinity
by $SO (\Gamma)$. If $f \in SO(\Gamma)$, then it follows from (\ref{e2}) that
all limit operators $f^h I$ of $fI$ are operators of multiplication by a
$\sZ^n$-periodic function $f^h$, that is $V_\beta f^h = f^h$
for every $\beta \in \sZ^n$.

We consider the operator $A = aI + b \cS_{\Gamma, \phi}$ under the assumptions
that $a, \, b \in PC^\infty (\Gamma) \cap SO(\Gamma)$ and $\phi \in
C^\infty(\Gamma \times \sR^2)$, and we suppose that the function $(x, \, z)
\mapsto \phi(x, \, z)$ is slowly oscillating with respect to $x \in \Gamma$
uniformly with respect to $z$ on compacts subsets of $\sR^2$. Under these
conditions, the operator $A$ is rich, and all limit
operators of $A$ are of the form $A^h =a^h I + b^h \cS_{\Gamma, \phi}^h$ where
\[
(\cS_{\Gamma, \phi}^h u) (x) = \frac{1}{\pi i} \int_\Gamma
\frac{\phi^h (x, \, x - y) u(y)}{y-x} dy, \quad x \in \Gamma,
\]
and $\phi^h$ is the limit function of $\phi$ with respect to $h \in \cH$ in the
sense that
\[
\phi^h (x, \, z) = \lim_{m \to \infty} \phi (x + h(m), \, z)
\]
uniformly on compact subsets of $\Gamma \times \sR^2$. Since $\phi$ is
slowly oscillating with respect to the first variable, the function $\Gamma
\times \sR^2 \ni (x, \, z) \to \phi^h (x, \, z)$ is periodic with respect to the
shifts $V_\alpha$ with $\alpha \in \sZ^n$. Consequently, the operator
$\cS_{\Gamma, \phi}^h$ is invariant with respect to these shifts.

Note that the operator $\tilde{A}^h := U A^h U^{-1}$ is of the form
\[
(\tilde{A}^h u) (x, \, \alpha) = a^h(x)I + b^h(x) \sum_{\beta \in \sZ^n}
\int_{\Gamma_0} k^h (x, x - y + \alpha -\beta) u(y, \, \beta) dy
\]
where $(x, \, \alpha) \in \Gamma_0 \times \sZ^n$ and
\[
k^{} (x, \, z) = -\frac{1}{\pi i} \frac{\phi^h (x, \, z)}{z}, \quad (x, \, z)
\in \Gamma_0 \times \sR^2.
\]
We associate with $\tilde{A}^h$ the operator-valued function
\[
\mu (\tilde{A}^h) : \sT^2 \to \cB(L_w^p (\Gamma_0)), \quad
\tau \mapsto a^h I + b^h \cM^h(\tau)
\]
where
\begin{equation} \label{f1}
\cM^h(\tau) := \sum_{\gamma \in \sZ^n} M_\gamma^h \tau^\gamma, \quad \tau \in
\sT^n,
\end{equation}
and
\begin{equation} \label{f2}
(M_\gamma^h \varphi) (x) = \int_{\Gamma_0} k^h (x, \, x-y + \gamma)
\varphi(y) dy, \quad x \in \Gamma_0.
\end{equation}
Then the conditions
\begin{enumerate}
\item
$a^h (x) \pm b^h (x) \phi^h (x, \, 0) \neq 0$ for every point $x \in
\Gamma_0 \setminus \cV$,
\item
$\liminf_{r \to 0} \inf_{\lambda \in \sR} |\det (\tilde{a}^h (\omega) +
\tilde{\phi}^h (\omega, \, 0) \tilde{b}^h (\omega) \hat{\sigma}_\omega
(\cS_\Gamma)(r,\, \lambda))| > 0$ for every $\omega \in \Gamma_0 \cap \cV$,
\end{enumerate}
imply the Fredholm property of all operators $\mu (\tilde{A}^h)(\tau) \in
\cB(L_w^p (\Gamma_0))$ with $\tau \in \sT^n$.
\begin{theorem} \label{Th1}
Let $a, \, b \in PC^\infty (\Gamma) \cap SO(\Gamma)$. Then the operator $A = aI
+ b \cS_{\Gamma, \phi} : L_w^p(\Gamma) \to L_w^p (\Gamma)$ is a Fredholm
operator if and only if the following conditions are satisfied: \\[1mm]
$(a)$ $a(x) \pm b(x) \phi (x, \, 0) \neq 0$ for every $x \in \Gamma
\setminus \cV$; \\[1mm]
$(b)$ $\liminf_{r \to 0} \inf_{\lambda \in \sR} |\det (\tilde{a}(\omega)
+ \tilde{b}(\omega) \tilde{\phi}(x, \, 0) \hat{\sigma}_\omega (\cS_\Gamma) (r,
\, \lambda))| > 0$ for every $\omega \in \cV$; \\[1mm]
$(c)$ all limit operators $A^h$ of $A$ are invertible.
\end{theorem}
Note that conditions $(a)$ - $(c)$ are equivalent to the invertibility
of the symbol functions $\sigma_A^x$, $\sigma_A^\omega$ and
$\mu(\tilde{A}^h)$, respectively, at every point of their domain of definition.

\begin{proof}
First we remark that $A$ satisfies conditions A1 - A5. Condition $(a)$
is the condition for the ellipticity of the restriction of $A$ on the open
edge $e$ at the point $x \in e$. This condition is necessary and sufficient
for the local Fredholmness of $A$ at $x \in e$. Condition $(b)$ is necessary and
sufficient for the local invertibility (hence, for the local Fredholmness) of $A$
at the vertex $\omega$; see, for instance, \cite{RIzv} and Chapter 4 in
\cite{RRS}. The final condition $(c)$ is equivalent to the local invertibility
of $A$ at the point $\infty$. Hence, the assertion of the theorem follows from
Theorem \ref{te1}.
\end{proof}

For an example, let $A = aI + b \cS_{\Gamma, \phi}$ and assume that
$\lim_{\Gamma \ni x \to \infty} b(x) = 0$. Then condition $(c)$ in Theorem
\ref{Th1} is equivalent to the condition
\begin{equation} \label{e3}
\liminf_{\Gamma \ni x \to \infty} |a(x)| > 0.
\end{equation}
Hence $A$ is a Fredholm operator on $L_w^p(\Gamma)$ if and only if conditions
$(a)$ and $(b)$ from Theorem \ref{Th1} and condition (\ref{e3}) hold.
\subsection{Locally compact operators on graphs}
Let again $\Gamma$ be a metric graph which is embedded into $\sR^2$ and
$\sZ^n$-periodic with $n \in \{1, \, 2\}$. We consider integral operators $A =
aI + bT$ where $a, \, b \in BUC(\Gamma) \cap SO(\Gamma)$, with $BUC(\Gamma)$ the
space of the bounded and uniformly continuous functions on $ \Gamma$, and $T$ is
an integral operator of the form
\[
(Tu)(x) = \int_\Gamma k(x-y) u(y) dy, \quad x \in \Gamma,
\]
where $k: \sR^2 \to \sC$ is a continuous function with the property that there
exist $C > 0$ and $\varepsilon > 0$ such that
\begin{equation} \label{eneu3}
|k(z)| \le C(1 + |z|)^{-2 - \varepsilon} \quad \mbox{for all} \;
z \in \sR^2.
\end{equation}
This estimate implies that $A \in W(\Gamma)$ and, consequently, $A$ is a bounded
operator on $L^p(\Gamma)$ for every $p \in [1, \, \infty]$. Moreover, the
boundedness of $k$ implies the local compactness of $T$ on every $L^p(\Gamma)$
with $1 < p < \infty$, that is, the operators $T\chi_MI$ and $\chi_M T$
are compact for every compact subset $M$ of $\Gamma$. Hence, the operator
$A$ is of local type on $L^p(\dot{\Gamma})$.

If $a \in BUC(\Gamma)$ then, by the Arzela-Ascoli theorem, every sequence $h \in
\cH$ has a subsequence $g$ such that the sequence of the functions $a(\cdot +
g(m))$ converges to a limit function $a^g$ uniformly on $\Gamma_0$. Hence, $A$
is a rich operator, and every limit operator of $A$ is of the form $A^g = a^g I
+ b^g T$ where $a^g, \, b^g$ are limit functions of $a, \, b$ with
respect to $g \in \cH$. Because of $a, \, b\in BUC (\Gamma) \cap SO(\Gamma)$,
the functions $a^g, \, b^g$ are periodic with respect to the action of $\sZ^n$.
Hence the limit operators $A^g$ are invariant with respect to shifts $V_\alpha$
with $\alpha \in \sZ^n$.

As above, we associate with $A^g$ the operator-valued symbol
\[
\mu_{A^g} : \sT^2 \to \cB(L^p(\Gamma_0)), \quad \tau \mapsto a^gI + b^g
\cM(\tau)
\]
where
\[
(\cM(\tau) u)(x) = \sum_{\alpha \in \sZ^n} \left( \int_{\Gamma_0} k(x - y
+\alpha) u(y) dy \right) \tau^\alpha, \quad x \in \Gamma_0, \, \tau \in
\sT^n.
\]
\begin{theorem} \label{th1}
Under the above assumptions for $a, \, b$ and $k$, the operator $A = aI + bT$ is
a Fredholm operator on $L^p (\Gamma)$ for $1 < p < \infty$ if and only if the
following conditions are satisfied: \\[1mm]
$(a)$ $\inf_{x \in \Gamma} |a(x)| > 0$; \\[1mm]
$(b)$ every limit operator $A^g$ of $A$ is invertible on $L^p(\Gamma)$.
\end{theorem}
Note that condition $(b)$ is equivalent to the invertibility of the operator
$\mu_{A^g}(\tau)$ on $L^p(\Gamma_0)$ for every $\tau \in \sT^n$.

\begin{proof}
Since $T$ is locally compact, the operator $A$ is locally Fredholm at the point
$x \in \Gamma $ if and only if it satisfies condition $(a)$. Condition $(b)$
is necessary and sufficient for the local invertibility of $A$ at the point
$\infty$. Hence Theorem \ref{th1} follows from Simonenko's local
principle (Proposition \ref{p1}).
\end{proof}
{\small Authors' addresses: \\[3mm]
Vladimir S. Rabinovich, Instituto Polit\'{e}cnico Nacional, \\
ESIME-Zacatenco, Av. IPN, edif. 1, M\'{e}xico D.F., 07738, M\'{E}XICO. \\
e-mail: vladimir.rabinovich@gmail.com \\[1mm]
Steffen Roch, Technische Universit\"{a}t Darmstadt, \\
Schlossgartenstrasse 7, 64289 Darmstadt, Germany.\\
e-mail: roch@mathematik.tu-darmstadt.de}
\end{document}